%% file: arxiv.tex
\documentclass[informs-stsy]{arxiv}              
\usepackage[numbers]{natbib}

\usepackage{amsthm}

\usepackage{hyperref}
\hypersetup{
    colorlinks   = true,
    linkcolor = blue,
     citecolor    = blue
}

\usepackage[utf8]{inputenc}
\usepackage[T1]{fontenc}
\usepackage[english]{babel}
\usepackage{appendix}
\usepackage{enumitem}
\usepackage[dvipsnames]{xcolor}
\usepackage{amsmath,amssymb}
\usepackage{xfrac}
\usepackage{cancel}
\usepackage{bbm}
\usepackage{bm}
\usepackage{mathtools}

\usepackage{subcaption}
\usepackage{multirow} 
 \usepackage{hyperref}  

\newtheorem{theorem}{Theorem}
\newtheorem{proposition}{Proposition}

\newtheorem{lemma}{Lemma}
\newtheorem{corollary}{Corollary}
\numberwithin{equation}{section}

\allowdisplaybreaks

\newcommand\E{\mathbb{E}}
\newcommand\F{\mathcal{F}}
\newcommand\X{\mathcal{X}}
\newcommand\T{\mathcal{T}}
\newcommand\N{\mathcal{N}}

\newcommand\Leq{\preceq}

\newcommand\x{\bm{x}}
\newcommand\y{\bm{y}}
\newcommand\e{\bm{e}} 
\newcommand\Exp[1]{\mathsf{Exp}\left(#1 \right)}

\newcommand\two{\{1,2\}} 
\newcommand{\revision}[1]{\textcolor{blue}{#1}}
\renewcommand{\revision}[1]{#1}

\newcommand{\arxiv}[1]{{\leavevmode\color{red}#1}}
\renewcommand{\arxiv}[1]{#1}

\newcommand{\journal}[1]{}

\newcommand\defeq{:=} 

\newcommand\q{\bm{q}} 
\newcommand\p{\bm{p}} 
\newcommand\s{\bm{s}} 
\newcommand\w{\bm{w}} 
\newcommand\Z{\mathbb{Z}} 

\newcommand\deq{{\buildrel d \over =}} 
\newcommand{\I}[1]{\mathbb{I}_{\left\{#1 \right\}}}
\renewcommand{\P}[1]{\mathbb{P}\left(#1 \right)}
\newcommand{\Ppi}[1]{\mathbb{P}_\pi\left(#1 \right)}
 
\newcommand{\TV}[2]{d_{\mathrm{TV}}\left(#1,#2\right)}
\newcommand\R{\mathbb{R}}
\renewcommand\wp{\mathrm{w.p.}}

\newcommand\tvratet{O \left( \frac{1}{(1-\rho)^3} \frac{1}{t} \right)}
\newcommand\lengthrate{O\left(
\frac{
1
}{(1- \rho)^3}
\frac{1}{\sqrt{t}} \right)}

\newcommand{\set}[1]{\left\{ #1 \right\}}
\renewcommand{\mc}[1]{\mathcal{#1}}
\DeclarePairedDelimiter\ceil{\lceil}{\rceil}

\makeatletter
\newcommand{\shorten}{\@ifnextchar\bgroup{\@shorten}{\@shorten}}
\long\def\@shorten#1{\begingroup\color{red}#1\endgroup}
\makeatother

\begin{document}


\TITLE{Convergence Rate   of the Join-the-Shortest-Queue System} 

\ARTICLEAUTHORS{%
	\AUTHOR{Yuanzhe Ma}
	\AFF{Department of Industrial Engineering and Operations Research, Columbia University, \texttt{ym2865@columbia.edu}}
	\AUTHOR{Siva Theja Maguluri}
	\AFF{Department of Industrial and Systems Engineering, Georgia Institute of Technology, \texttt{siva.theja@gatech.edu}} }

\ABSTRACT{
\input{abstract}
}

\KEYWORDS{Join-the-Shortest-Queue; Load Balancing System; Transient Analysis; Coupling}


\maketitle

\input{introduction}

\input{model_arxiv}

\input{proof_lemma_1}

\input{proof_lemma_dominance_jsq}
\input{proof_lemma_T_x}

\input{proof_lemma_jsq_hitting_heavy_traffic}

\input{proof_improved_bound}
\input{proof_lemma_JSQ_T0_T1}

\input{proof_of_corollary}
\input{conclusion}

\section*{Acknowledgments}

This work was partially supported by NSF grants EPCN-2144316 and CMMI-2140534.

\bibliographystyle{abbrvnat}
\bibliography{references}   


\end{document}

%% file: abstract.tex
The Join-the-Shortest-Queue (JSQ) policy is among the most widely used 
load balancing algorithms
and has been extensively studied. 
However, an
exact 
characterization of the system
behavior remains challenging.
Most prior research has focused on analyzing its performance in 
the steady state 
in certain asymptotic regimes,
such as the heavy-traffic regime.  
However,   convergence 
to the steady state in these regimes is often slow,  so steady-state and heavy-traffic characterizations may be less informative over practical time horizons.
To address this limitation, we provide a
finite-time convergence rate analysis of a JSQ system 
with two symmetric servers. 
In sharp contrast to the existing literature, we directly study the original system rather than  an approximate limiting system such as a diffusion approximation. 
Our results demonstrate that for such a system, 
the convergence rate to its steady state, measured in the total variation distance,
is   
$\tvratet$, 
where $\rho \in (0,1)$
is the traffic intensity.

%% file: introduction.tex
 \section{Introduction} \label{sec:intro}

Efficient load balancing is central
to
optimizing resource allocation and minimizing job delays in 
modern service systems and data centers.
The Join-the-Shortest-Queue (JSQ) policy is
a popular 
load balancing algorithm
due to its simplicity and strong performance
under various criteria.
Since its introduction by~\citet{Winston77},
it has received significant attention over  the decades. 
However, the analysis of the JSQ system remains highly nontrivial~\citep{load_balancing_survey}.
Due to the complexity of the problem, most
prior works    consider a JSQ system 
only in   steady state.
Further,    
researchers typically conduct analyses in some 
asymptotic regimes,
such as  the heavy-traffic regime.
For example, there are various bounds on
the queue length of a JSQ system in
its steady state, such as~\cite[Theorem 10.2.3]{SrikantYi14}.
However,  
traffic patterns in modern
data centers tend to be time-varying~\cite{datacenter_traffic}.
This necessitates
a careful finite-time analysis of JSQ performance, 
rather than relying only on steady-state behavior.
Motivated by this,   
we focus on analyzing the \textit{transient behavior} 
of a JSQ system.
As a starting point,  we consider
a system consisting of two servers with 
both interarrival and service times
following exponential distributions.
We analyze its non-asymptotic convergence rate toward its steady state.

As discussed, prior works focus on analyzing the JSQ policy under the heavy-traffic regime and in its steady state. \arxiv{However, unsurprisingly,
simulation results in Figure~\ref{fig:JSQ-sim} reveal that the convergence rate of a JSQ system can be quite slow when the traffic intensity  
$\rho$
approaches 1.
Thus, to conduct
 more careful analysis
of the JSQ systems, 
we need to understand their transient behavior.}Many prior results rely on state space collapse phenomena in heavy traffic and do not directly yield explicit finite-time bounds for light traffic.
In this paper, we 
address this challenge
and obtain an explicit  upper
bound on the  
difference between the transient 
distribution of a JSQ system and its steady state for any 
traffic intensity
$\rho \in (0,1)$.

\arxiv{
\begin{figure}[htbp]
\centering
\begin{minipage}[b]{0.45\textwidth}
\centering
\includegraphics[width=\textwidth]{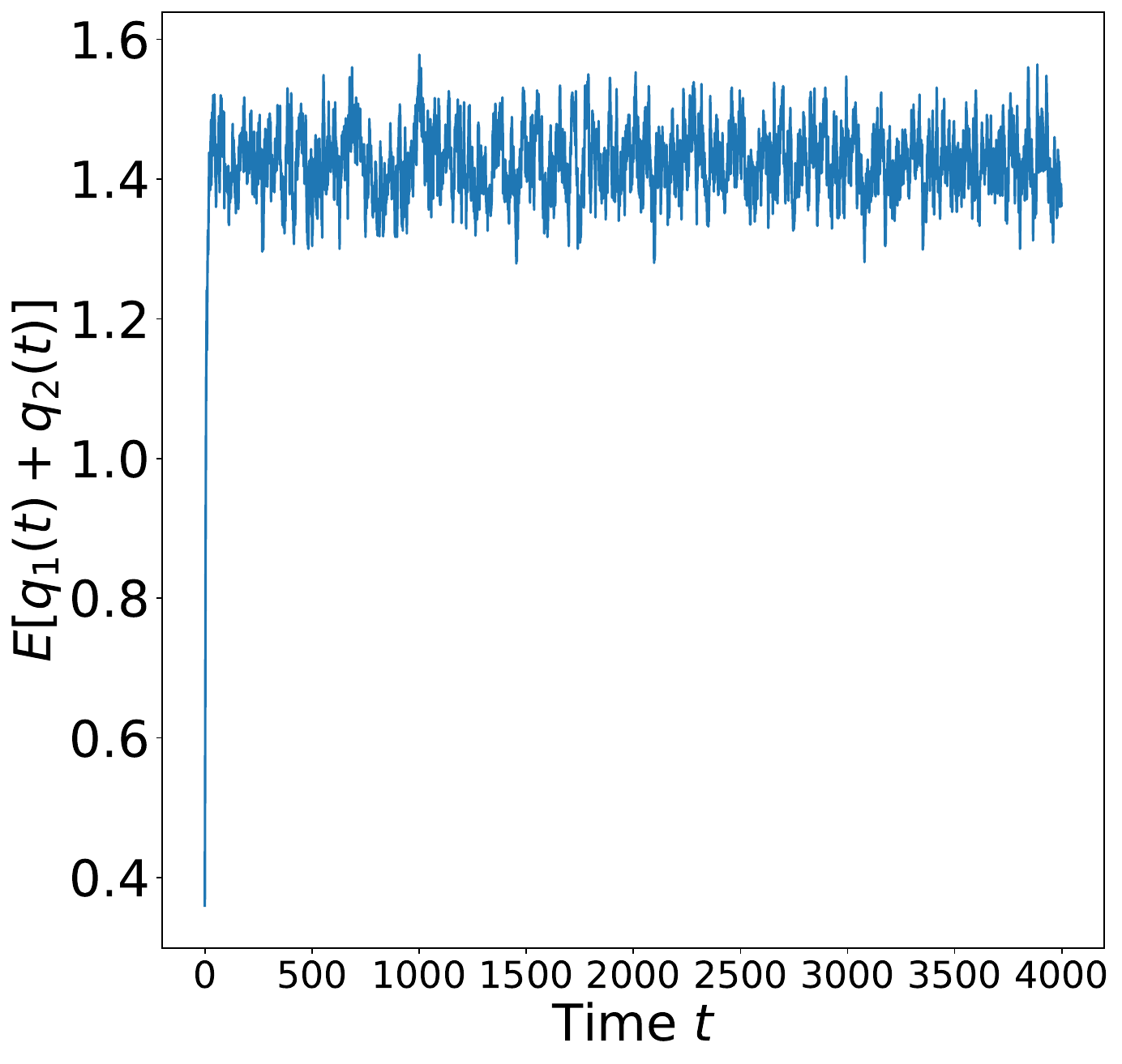}
\end{minipage}
\hfill
\begin{minipage}[b]{0.45\textwidth}
\centering
\includegraphics[width=\textwidth]{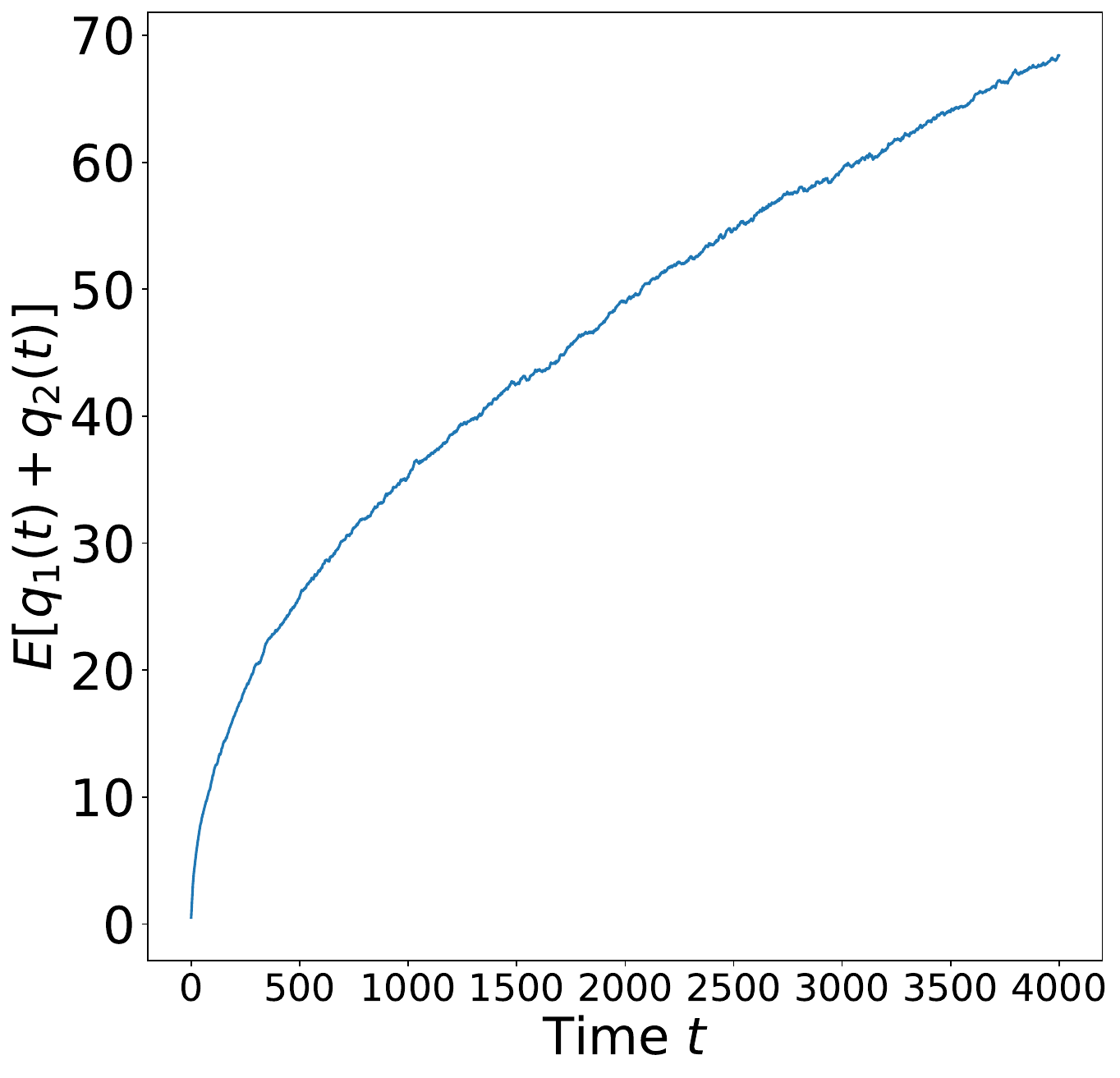}
\end{minipage}
\caption{Finite-time performance of a JSQ system at different values of $\rho$, where  $q_i(t)$  denotes the   queue length
at time $t$ for server $i$.
The results are based 
on 1000 simulations of trajectories 
to estimate the expected total queue length.
We consider a JSQ system with a Poisson arrival process of rate $\lambda = \rho$ and two identical servers, each with exponentially distributed service times of rate $\mu = 0.5$. Results are shown for $\rho = 0.5$ (left) and $\rho = 0.999$ (right).
The pair of plots clearly shows
that the convergence rates of JSQ systems heavily depend on $\rho$.
}
\label{fig:JSQ-sim}
\end{figure}
}

We focus  on analyzing the  
distribution $\mathbb{P}^t_{\x}$
of the queue length vector
at time $t$
for a JSQ system with an
initial queue length vector $\x$.
In particular,
we consider its total variation (TV) distance
to the steady state $\pi$ at every time $t \ge 0$.
The total variation  distance is 
a common metric quantifying the 
convergence rate of 
a stochastic system
and 
is
known to be closely connected to  many other metrics. 
For example, bounds on TV distances imply bounds on 
  Hellinger distances.
In addition to the TV distance, we derive
the convergence rate  of the mean     queue length.
In our coupling approach,
the key challenge  is to 
obtain an upper bound on the expected hitting times  or  
first passage times of the JSQ system.
For general continuous-time Markov chains (CTMCs), 
analyzing the hitting time
usually   requires
  extensive exploitation  of 
the structure of the  CTMC~\citep{DiCrescenzoMa08}.
We use several structural properties  of the JSQ system
and show
that the problem can be reduced to providing an upper bound on $\E[T_0]$, 
where $T_0$, to be 
defined in~\eqref{eqn:T0-def}, is the time it takes for 
a JSQ system's queue length vector
to hit the origin starting from state $(1,0)$. 

\subsection{Main contribution}

Our main result is for a continuous-time JSQ system with
two symmetric queues, where
both interarrival and service times 
follow
exponential distributions. 
We derive an upper bound on  $\TV{\mathbb{P}^t_{\x}}{\pi}$
and, as a corollary, an upper bound 
on the difference between the mean queue length at time 
$t$ and its steady-state value.
These two   quantities  converge  to zero 
as $t \to \infty$ for any 
traffic intensity
$\rho \in (0,1)$. 
In particular, for the TV distance, 
we establish 
in   Theorem~\ref{thm:main} 
that the convergence rate is $\tvratet.$
In contrast to prior work   based on
diffusion limits, we derive transient bounds by directly   studying the original JSQ system, 
thus   avoiding approximation errors.
To the best of our knowledge, we are the first to provide 
a convergence rate analysis
of a JSQ system for any $\rho \in (0,1)$.

The rest of the paper is organized as follows.
We first discuss related work in Section~\ref{sec:related-work} 
and notation in Section~\ref{sec:notation}.
We present the model and state 
our main result
in Section~\ref{sec:model-and-result}.
We provide all   proofs in Section~\ref{sec:proof-of-thm}.
Finally, we provide concluding remarks and promising future directions in Section~\ref{sec:conclusion}.

\subsection{Related work}
\label{sec:related-work}
 
Since the pioneering work of~\citet{Winston77}, 
the JSQ policy 
has been extensively studied in the literature~\citep{Conolly84,load_balancing_survey}.
Transient analyses
are generally 
complicated, 
even for simple systems such as $M/M/1$ queues~\cite{Takacs62,AbateWh87}
and $M/M/n$ queues~\cite{GamarnikGo13}.
There are   numerical results for
transient analysis of JSQ systems~\cite{Grassmann80}, 
but few theoretical results are known regarding their convergence rates.    
In this work, we analyze the transient behavior through a hitting time analysis 
(see, e.g.,~\cite{LevinPe17}). 
Though
there are some   results related to   
hitting times or first passage times for 
JSQ systems~\citep{YaoKn08,Tibi19},    
they cannot be used directly
 for our problem. 
For a simple system such as an $M/M/1$ system, 
one can obtain
the
moment generating function (MGF)
  of the hitting time for the queue length
  to hit state $0$   
 using a construction of an exponential martingale and the optional stopping theorem~\cite[Proposition 5.4]{Robert13}. 
One of the key reasons this works is     the special property of an $M/M/1$ queue: starting from state $x \ge 1$, the queue length must first hit states 
$x-1,x-2,\cdots,1$ before hitting state 0. 
However, this property is not true for a JSQ system. 
For example, consider a JSQ system with 
an initial queue length
$(x,y)$, where $x,y \ge 1$.
It is possible that the queue length vector never hits state $(x-1,y-1)$ before hitting $(0,0)$. 
This is one of the technical barriers 
to analyzing a JSQ system.
One possible approach is to use the exponential martingale constructed in~\cite{Tibi19};
however, that work focuses on the hitting time to $(K,K)$, $K \ge 1$, 
whereas our interest is in the hitting time to $(0,0)$, which is fundamentally different, making those results inapplicable.

 The convergence rates of stochastic systems 
 and hitting time analyses based on 
 Foster–Lyapunov drift functions have been extensively studied in the literature; see, e.g.,~\cite{DownMeTw95}.
For example,  
Theorem 2 of~\cite{LundMeTw96}
provides a  general framework for obtaining 
exponential convergence rates of 
general stochastic systems.  
However, applying it requires a suitable drift function, 
and it is unclear how to construct one for the JSQ system.
 Relatedly, 
 researchers have discussed
 transient analysis for queues
under various 
limiting regimes or for 
simpler systems~\cite{vanLeeuwaardenKn11, GamarnikGo13}.   
In contrast to these works,
we   analyze the JSQ system in a 
fixed system setting.

\subsection{Notation}
\label{sec:notation}

In this section, we introduce the notation that will be used 
throughout this paper. 
We use $\P{A}$ to denote the probability that event $A$ occurs, 
and use $\E[X]$ to denote the expected value of the random variable $X$.
For a vector $\x$, we use $x_i$ to denote its $i$-th element.
Given two vectors $\x$ and $\y$ in $\R^n$, we write $\x \geq \y$ if $x_i \geq y_i$ for every $i \in \{1,\ldots,n\}$. 
For a stochastic process $(X_t)_{t \geq 0}$ with 
steady state $\pi$, we use
$\E_\pi[X]$ 
and $\mathbb{P}_\pi(\cdot)$
to denote the expectation and   distribution
of $X_t$ with respect to the steady state. 
For two random variables $A$ and $B$, we write that 
$A$ and $B$  are
  equal   in
distribution, or $A \deq B$, if $\P{A \leq x} = \P{B \leq x}$ for all $x \in \R$. We define $A$ to be  stochastically dominated by $B$, or $A\preceq B$, if and only if $\P{A > x} \leq \P{B > x}$ for all $x \in \R$.
For two probability distributions 
$\mathbb{P}$ and $\mathbb{Q}$ defined on the sigma-algebra $\F$, we denote the total variation distance between $\mathbb{P}$ and $\mathbb{Q}$ as
$\TV{\mathbb{P}}{\mathbb{Q}} = \sup_{A \in \F} |\mathbb{P}(A) - \mathbb{Q}(A)|.$

%% file: model_arxiv.tex
\section{Model and results}
\label{sec:model-and-result}

\subsection{Model}
\label{sec:model}

 Throughout the paper,
we consider a 
continuous-time
load balancing system operating under the  JSQ
policy.
It consists of two single-server queues,
and each has an infinite buffer size.
Jobs arrive
according to a  Poisson process with rate $\lambda$,
and   service times for  each queue are independent exponential random variables
with mean $\frac{1}{\mu}$. 
We use $q_i(t), i \in \two$ to denote the number of jobs in the $i$-th queue at 
time $t$. 
Let $\q(t) \in \Z_+^2 = \set{0,1,2,3\cdots}^2$ be the queue length process with elements $(q_1(t),q_2(t))$. 
The JSQ policy  
sends each arriving job to the queue
with 
the fewest
jobs, and picks
the one with the smaller index when there is a tie. 
We define 
$i^\star(\bm{q}) = \min\set{ \underset{i \in \two}{\argmin}   \set{q_i} }.$
The queue length process $\{\q(t): t \geq 0 \}$ is a  CTMC with generator matrix $G$ defined as 
\begin{align*}
G_{\q,\q'}  =  
\begin{cases}
\mu  & \mathrm{if~} q_i > 0 \mathrm{~and~} \q' = \q - \e^{(i)} \mathrm{~for~}  i \in \two, \\  
 \lambda  & \mathrm{if~}  \q' = \q + \e^{(i^\star(\q))},    \\  
-(\lambda  + \mu \sum_{i=1}^2 \I{q_i > 0} )  & \mathrm{if~} \q=\q',   \\
 0 & \mathrm{otherwise,}   \\  
 \end{cases}
\end{align*} 
where $\e^{(i)}$ is the $i$-th unit vector. 
 

\subsection{Main result}
 
We now present the main result  of this paper:
the convergence rate for the
 JSQ system   described in  
 Section \ref{sec:model}.


 \begin{theorem} 
\label{thm:main}
Consider a JSQ system    as  described in Section~\ref{sec:model}
with 
$\rho = \frac{\lambda}{2\mu} < 1$.
Let its initial queue length vector be   $\x$,
and
let 
$\mathbb{P}_{\x}^t$ be the  distribution 
of its queue length vector $\q(t)$
at time $t$, and let $\pi$ be its steady state. 
Then, for all $t > 0$,
\begin{align} 
\label{eqn:TV-bound}
\TV{\mathbb{P}_{\x}^t}{\pi} \le
\frac{1}{t} C_1(\lambda,\mu,\x),
 \end{align}
 where 
 \begin{align} 
\label{eq:definition of C1}
C_1(\lambda,\mu,\x) = \frac{2-\rho}{2 \mu (1-\rho)^3}  \Big( 
(x_1+x_2)(1-\rho) +  \rho(2 - \rho)
\Big).
\end{align}
 Further, define 
\begin{align} 
\label{eqn:def-K}
K(\rho) 
= \begin{cases}
\min\set{1, 
\frac{2+ 2\rho}{1-2\rho^2} 
\cdot \frac{(1-\rho)^2} {2- \rho}},  & \mathrm{~if~} \rho \in[0,  \frac{1}{\sqrt{2}}), \\
1,  & \mathrm{~if~}  \rho \in [\frac{1}{\sqrt{2}},1),
\end{cases}
\end{align}
and let
 \begin{align} 
\label{eq:definition of C1-K}
C_1^K(\lambda,\mu,\x) = C_1(\lambda,\mu,\x)   \cdot K(\rho).
\end{align}
Then, for all $t > 0$,
\begin{align} 
\label{eqn:TV-bound-stronger}
\TV{\mathbb{P}_{\x}^t}{\pi} \le
\frac{1}{t} C_1^K(\lambda,\mu,\x).
 \end{align}
 \end{theorem}

\begin{figure}[h!]
\centering
\begin{subfigure}[b]{0.45\textwidth}
\centering
\includegraphics[height=6cm]{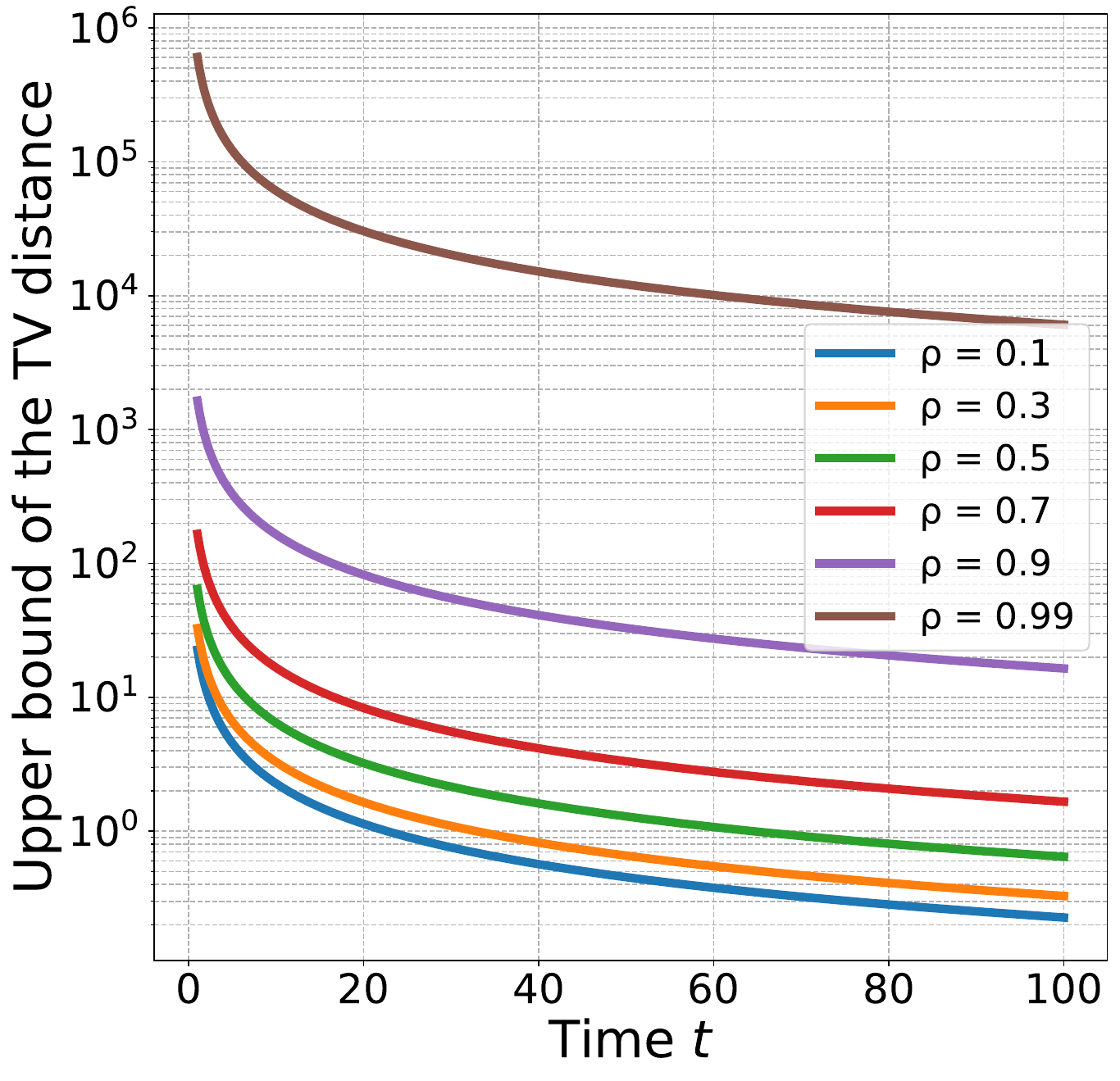}   
\end{subfigure}
\hfill
\begin{subfigure}[b]{0.45\textwidth}
\centering
\includegraphics[height=6cm]{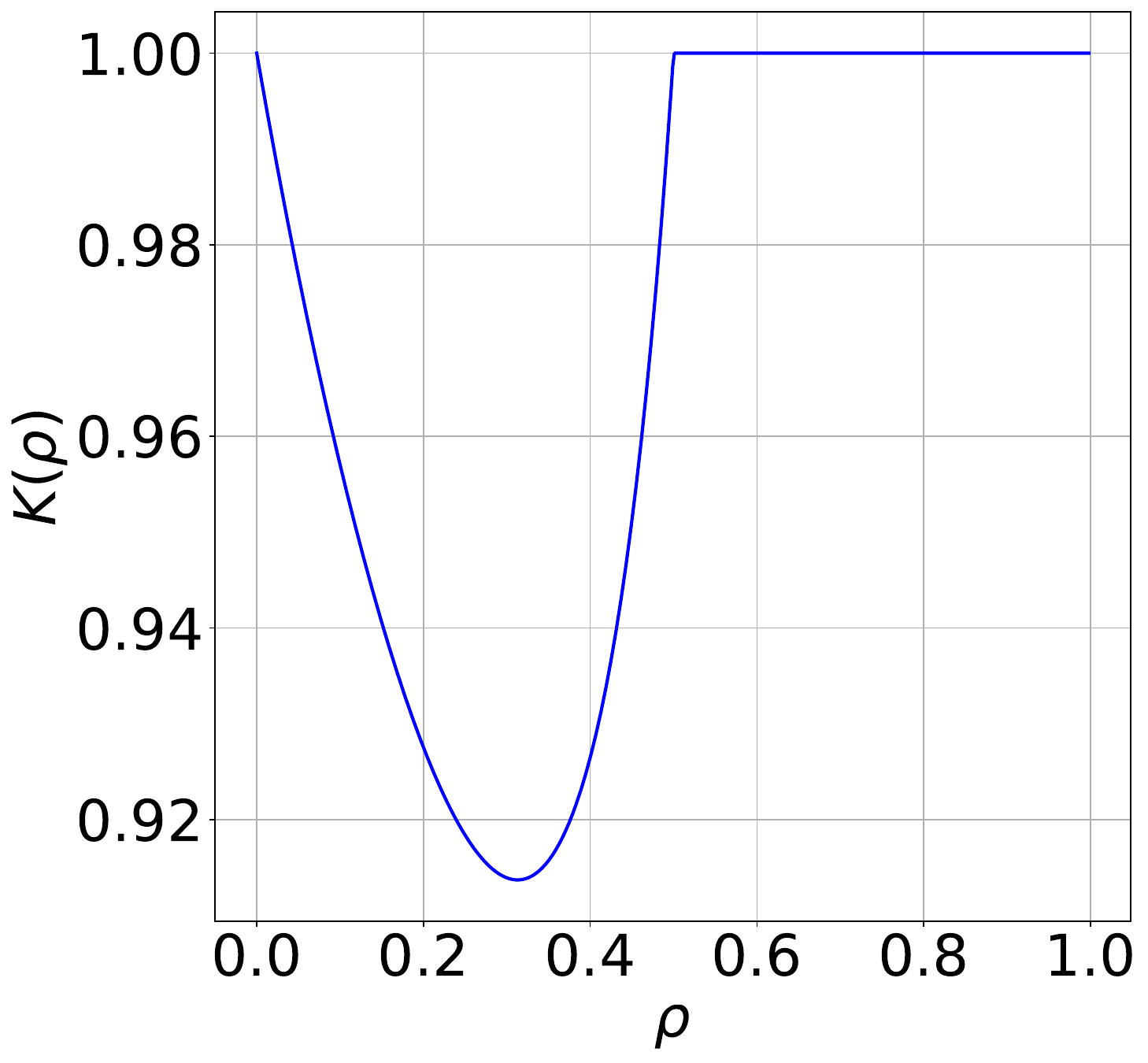}   
\end{subfigure}
\caption{Left: plot of the bound~\eqref{eqn:TV-bound-stronger} as a function of $t$
for different values of $\rho$
and $x_1= x_2=10, \mu  = 1$.
Right: plot of $K(\rho)$~\eqref{eqn:def-K} as a function of $\rho$.}
\label{fig:bound}
\end{figure}


 We now visualize the upper bound~\eqref{eqn:TV-bound-stronger} 
 and the value of $K(\rho)$ \eqref{eqn:def-K} 
in Figure~\ref{fig:bound}.
As we observe, the right-hand side of~\eqref{eqn:TV-bound-stronger}
varies
significantly  across $\rho$, 
though its dependence on $t$ is the same for all $\rho$.

Next, in Corollary~\ref{coro-moment-difference}, we establish a convergence rate result
for the difference between the transient and steady-state mean queue lengths.
In particular, the bound~\eqref{eqn:moment-diff-bound}
implies that the mean queue length converges at 
rate $\lengthrate$; 
the proof is provided in Section \ref{sec:proof-coro}.
Combining this corollary with existing results on the stationary queue length of the JSQ system, such as Proposition~4 of~\cite{Dester17}, 
yields an explicit transient upper bound on the mean queue length, as stated in~\eqref{eqn:moment-bound}.


\begin{corollary} 
\label{coro-moment-difference}
 Consider a JSQ system    as  described in Theorem~\ref{thm:main} with
 $\rho = \frac{\lambda}{2\mu} < 1$
 and an 
 initial queue length vector     $\x$.
Then
\begin{align}
|\E[q_1(t)+q_2(t)] - \E_\pi[q_1 + q_2]|
\le \frac{1}{\sqrt{t}}  \Big(2C_1^K(\lambda,\mu,\x) + C_2(\lambda,\mu,\x)   \Big) + \frac{2}{t} C_1^K(\lambda,\mu,\x)
\label{eqn:moment-diff-bound}
\end{align}
and
\begin{align}
\E[q_1(t)+q_2(t)]  \leq 
\frac{\rho(2 - \rho)}{1-\rho} +\frac{1}{\sqrt{t}}  \Big(2C_1^K(\lambda,\mu,\x) + C_2(\lambda,\mu,\x)  
\Big) + \frac{2}{t} C_1^K(\lambda,\mu,\x),
\label{eqn:moment-bound}
\end{align}

where $C_1^K(\lambda,\mu,\x)$ is defined in~\eqref{eq:definition of C1-K}
and 
\begin{align} 
\label{eq:definition of C2}
C_2(\lambda,\mu,\x) =
(x_1^2+ x_2^2)  + 2(x_1+x_2)  
\frac{\rho}{1-\rho}
+ \frac{4\rho(1+\rho)}{(1-\rho)^2}. 
\end{align} 
\end{corollary}

\section{Proof of Theorem~\ref{thm:main}}
\label{sec:proof-of-thm}

In the proof, 
we   use the following hitting time, denoted by $T_0$ \eqref{eqn:T0-def}. 
Consider a JSQ system  described in Section~\ref{sec:model}
with queue length process $\q(t)$. We define 
\begin{align}
T_0 & = \inf\set{ t \ge 0:  \q(t) =(0,0) \mid \q(0) = (1,0) } 
\label{eqn:T0-def}
\end{align} as the time for the queue length   process to
hit state $(0,0)$ starting from $(1,0)$.
In the proof, we  
use two coupled processes defined in Section~\ref{sec:couple-definition}.

\subsection{Coupling with arrival and service} \label{sec:couple-definition}

Consider two
JSQ systems    
with queue length vectors $\p(t)$ and $\q(t)$ as described in Section~\ref{sec:model}. 
We call them 
coupled with arrival and service if they can be constructed using   
$\N_a$ and $\N_d^i$, $i =1,2$, where
$\N_a$ is a Poisson process with rate $\lambda$,
and $\N_d^1, \N_d^2$ are 
two independent Poisson processes, each with rate $\mu$.
We allow the initial queue length vectors of $\p$ and $\q$ to be different.
We also assume that 
$\N_a$, $\N_d^1$, and $\N_d^2$ are independent of each other. 
We interpret $\N_a$ as the arrival process 
and $\N_d^1, \N_d^2$ as the 
potential service processes. 
Accordingly, we can write the 
 expressions for the
dynamics of the queue length vectors in the two 
JSQ systems as
\begin{align}
\begin{split}
dp_1(t) = \I{p_1(t-) \leq p_2(t-)} \N_a(dt) - \I{p_1(t-) \geq 1} \N_d^1(dt), \\
dp_2(t) = \I{p_1(t-) > p_2(t-)} \N_a(dt) - \I{p_2(t-) \geq 1} \N_d^2(dt), \\
dq_1(t) = \I{q_1(t-) \leq q_2(t-)} \N_a(dt) - \I{q_1(t-) \geq 1} \N_d^1(dt), \\
dq_2(t) = \I{q_1(t-) > q_2(t-)} \N_a(dt) - \I{q_2(t-) \geq 1} \N_d^2(dt),  
\end{split}
\label{eqn:couple-def}
\end{align}
where $p_1(t-)$ is     
the value of
$p_1$ evaluated just before time $t$.

\subsection{Partial proof of Theorem~\ref{thm:main}}
\label{sec:proof-thm-1}


We first prove a weaker version of Theorem~\ref{thm:main}:
we only prove it up to \eqref{eqn:TV-bound}.
The proof 
 builds on
Proposition~\ref{prop-TV-coupling} and Lemma~\ref{lemma-JSQ-T-prob}.
We present the complete  proof of
Theorem~\ref{thm:main}
in Section~\ref{sec:improved-K-rho}.


\begin{proposition}~\cite{LevinPe17, Lindvall94}
\label{prop-TV-coupling}
Let $(X_t)_{t \ge 0}$ and   $(Y_t)_{t \ge 0}$
be  two copies of the same
continuous-time, discrete-valued Markov chain on a common probability space with $X_0 = \x$
and $Y_0 \sim \pi$,
where  $\pi$
is the steady state of this Markov chain.
We   assume $X_t$ and $Y_t$ take values in a countable set $\X$,
and 
let $T = \inf\{t  \ge 0: X_t = Y_t \}$ denote the first time that they meet.
Then
\begin{align*}
\TV{\mathbb{P}_{\x}^t}{\pi} 
\leq \P{T \ge t}.
\end{align*}
\end{proposition}

\begin{lemma} \label{lemma-JSQ-T-prob}
Consider two  coupled  JSQ systems 
with queue length processes $\p(t)$ and  $\q(t)$
as governed by~\eqref{eqn:couple-def}.
 Assume $\p(0)=\x$  and  $\q(0) \deq \pi$, where $\pi$ is the 
 steady state of $\p$. 
 Let $T$ be the first time 
 that 
 the two systems   have the same queue length vector. 
We have
\begin{align}
\P{T \ge t} \le
\frac{1}{t} C_1(\lambda,\mu,\x),
\label{eqn:P-ge-t-upper}
\end{align}
where $C_1(\lambda,\mu,\x)$ is defined in~\eqref{eq:definition of C1}.
\end{lemma}

Now we are ready to present the proof of Theorem~\ref{thm:main} up to \eqref{eqn:TV-bound}.
\begin{proof}[Proof of Theorem~\ref{thm:main} up to \eqref{eqn:TV-bound}]
Applying Proposition~\ref{prop-TV-coupling} 
yields
\begin{align*}
 \TV{\mathbb{P}_{\x}^t}{\pi}  \leq \P{T \ge t},
\end{align*}
where $T$ is defined in Lemma~\ref{lemma-JSQ-T-prob}.
 By~\eqref{eqn:P-ge-t-upper} in Lemma~\ref{lemma-JSQ-T-prob},
 \begin{align*}
 \TV{\mathbb{P}_{\x}^t}{\pi} 
 \leq \P{T \ge t}
 \leq \frac{1}{t} C_1(\lambda,\mu,\x)
 \end{align*}
and we complete the proof of  \eqref{eqn:TV-bound}.
\end{proof}

\subsection{Proof of Lemma~\ref{lemma-JSQ-T-prob}}



To prove Lemma~\ref{lemma-JSQ-T-prob}, we first establish 
a monotonicity property for coupled JSQ systems in
Lemma~\ref{lemma-dominance-JSQ}
with proof
  in Section \ref{sec:proof-lemma-dominance-JSQ}.
  Next, we discuss
Lemma~\ref{lemma-T-x}, 
proved in Section~\ref{sec:proof-lemma-T-x}, 
which provides an upper bound on the hitting time to the origin for a JSQ system starting from an arbitrary state $\x$, 
in terms of $T_0$ defined in~\eqref{eqn:T0-def}.
We then bound $\E[T_0]$ in Lemma~\ref{lemma-JSQ-hitting-time-heavy-traffic} with proof   given in
Section~\ref{sec:proof-lemma-JSQ-hitting-time-heavy-traffic}.
 Together, these results imply Lemma~\ref{lemma-JSQ-T-prob}.
 
%

\begin{lemma} \label{lemma-dominance-JSQ}
Consider two  coupled  JSQ systems~\eqref{eqn:couple-def} with  queue length processes $\p(t)$ and  $\q(t)$.
If $p_i(0) \ge q_i(0)$ for $i \in \two$, then $p_i(t) \ge q_i(t)$ for $i \in \two$ and
for every  $t \ge 0$. In particular, 
for any $i \in \two$ and $t \ge 0$, $p_i(t) = 0 \Rightarrow q_i(t) = 0$. 
Further,
define 
$r_i(t) = p_i(t)  - q_i(t) \ge 0$ and $r(t) = r_1(t)+r_2(t)$. Then $r(t) \le r(0)$ for any $t \ge 0.$
\end{lemma}

\begin{lemma} \label{lemma-T-x}
Consider a   JSQ system with queue length process $\q(t)$.
 Define the random variable $\T(\x)$  as the  
  time for $\q(t)$
  to first hit state $(0,0)$
  with $\q(0) = (x_1,x_2)$:
 \begin{align*}
 \T(\x) = \T(x_1,x_2) = \inf \{t \ge 0:
 \bm{q}(t) = (0,0) | \bm{q}(0)=(x_1,x_2)  \}.
\end{align*}
Then we have $\T(x_1,x_2) \Leq \sum_{i=1}^{x_1+x_2}  T_{0i}$, where $T_{0i}$ are independent 
copies of $T_0$~\eqref{eqn:T0-def} and the relation $\Leq$ is defined in Section~\ref{sec:notation}.
\end{lemma}

\begin{lemma} \label{lemma-JSQ-hitting-time-heavy-traffic} 
Consider $T_0$~\eqref{eqn:T0-def}   and assume $\rho = \frac{\lambda}{2\mu} < 1$.
Then
\begin{align} 
\E[T_0] \leq
\frac{1}{\lambda} \frac{\rho (2 - \rho)}{(1-\rho)^2}.
 \label{eqn:ET1-T0-bound-general}
\end{align}
\end{lemma}

%% file: proof_lemma_1.tex

To prove Lemma~\ref{lemma-JSQ-T-prob}, we also use
Proposition~\ref{prop-JSQ-stationary-bound} from the literature,
which provides an upper bound on
$\E_\pi[q_1 + q_2]$ for the JSQ system. 

\begin{proposition} [Proposition 4 and Remark 2 of~\cite{Dester17}] 
\label{prop-JSQ-stationary-bound}
For the JSQ system     described in Theorem~\ref{thm:main} 
with $\rho = \frac{\lambda}{2\mu} < 1$,  we have
\begin{align*}
\E_\pi[q_1+ q_2] \leq  \frac{\rho(2 - \rho)}{1-\rho}.
\end{align*}
\end{proposition}

\revision{
Although~\cite{Dester17} uses
uniform random tie-breaking, its total-queue-length bound remains valid here because the 
sorted queue-length process is independent of the tie-breaking rule.}

\begin{proof}[Proof of Lemma~\ref{lemma-JSQ-T-prob}] 
We consider a JSQ system with queue length process $\s$,
driven by   the same arrival and potential service processes as $\p$ and $\q$,
with $s_j(0) = \max\set{q_j(0),p_j(0)}, j \in \two$.
We couple the three systems $\p,\q,\s$ similarly to Section~\ref{sec:couple-definition}. 
 Formally, let 
$\s$
 evolve according to
\begin{align*}
\begin{split}
ds_1(t) &= \I{s_1(t-) \leq s_2(t-)} \N_a(dt) - \I{s_1(t-) \ge 1} \N_d^1(dt), \\
ds_2(t) &= \I{s_1(t-) > s_2(t-)} \N_a(dt) - \I{s_2(t-) \ge 1} \N_d^2(dt). 
\end{split}
\end{align*}

From Lemma~\ref{lemma-T-x}, the time $\s$ hits 0, denoted by $T'$, 
is stochastically dominated by
$\sum_{i=1}^{s_1(0) + s_2(0)} T_{0i}$,
where $T_{0i}$ are i.i.d. copies of $T_0$~\eqref{eqn:T0-def}. 
Since $\s(0)\ge \p(0)$ and $\s(0)\ge \q(0)$ componentwise, Lemma~\ref{lemma-dominance-JSQ}
(applied to the coupled pairs $(\s,\p)$ and $(\s,\q)$) implies that for all $t \ge 0$,
$\s(t)\ge \p(t)$ and $\s(t)\ge \q(t)$ componentwise. Thus, for any $t \ge 0$,
$\s(t)=(0,0)$ implies $\p(t)=(0,0)$ and $\q(t)=(0,0)$, and hence 
$T \le T'$ almost surely.
Since $T' \Leq \sum_{i=1}^{s_1(0) + s_2(0)} T_{0i}$, we have
\begin{align}
    \E[T]
& \le \E[T']
= \E_{\s(0)}[\E[T'] \mid \bm{s}(0)] \nonumber \\
&  \le \E_{\s(0)}[(s_1(0)+s_2(0)) \E[T_0]]
= \E_{\q(0)}\Big[\max\set{q_1(0),p_1(0)} + \max\set{q_2(0),p_2(0)}\Big] \E[T_0]. 
\label{eqn:E-T-bound-in-q}
\end{align}
Using $\max\set{x,y} \le x+y$ for $x,y \ge 0$ and    $\q(0) \deq \pi$,~\eqref{eqn:E-T-bound-in-q} implies
\begin{align}
\E[T] 
\le \E[T_0] \Big( 
 (x_1+x_2) + \E_\pi [q_1 + q_2] 
 \Big).
 \label{eqn:T-upper-case-1}
\end{align}

Using
$\E_\pi[q_1+ q_2] \leq  \frac{\rho(2 - \rho)}{1-\rho}$ from
Proposition~\ref{prop-JSQ-stationary-bound} and recalling $\E[T_0] \leq
\frac{1}{\lambda} \frac{\rho (2 - \rho)}{(1-\rho)^2}= \frac{2-\rho}{2\mu (1-\rho)^2}$
from~\eqref{eqn:ET1-T0-bound-general} in Lemma~\ref{lemma-JSQ-hitting-time-heavy-traffic}, we 
obtain $\E[T] \le  C_1(\lambda,\mu,\x)$~\eqref{eq:definition of C1}.
Applying   
Markov's inequality    $\P{T \ge t} \leq \frac{\E[T]}{t}$ for $t > 0$,  we obtain $\P{T \ge t} \le  
 \frac{1}{t} C_1(\lambda,\mu,\x)$
and  we complete the proof of Lemma~\ref{lemma-JSQ-T-prob}.
 \end{proof}




%% file: proof_lemma_dominance_jsq.tex
\subsection{Proof of Lemma~\ref{lemma-dominance-JSQ}}

\label{sec:proof-lemma-dominance-JSQ}

Consider the processes $\N_a(t), \N_d^1(t)$, and $\N_d^2(t)$ in~\eqref{eqn:couple-def}.
Fix a  sample path.
Let $t_1$ be the first time that either
an arrival or a potential
service completion
event occurs. 
In   words, $t_1$ is the smallest time 
$t$ such that 
$\N_a(t)  =1$ or $\N_d^1(t) = 1$ or $\N_d^2(t) = 1$. We will show that $p_i(t_1) \ge q_i(t_1)$.
\begin{enumerate}
    \item If the first event is a potential service completion, then as we assume the two systems are coupled,  we   have
$dp_j(t_1) = - \I{p_j(0) \ge 1}, 
dq_j(t_1) = - \I{q_j(0) \ge 1}$, 
where $j$ is the index such that $\N_d^j(t_1) = 1$.
It is easy to verify that
$p_i(t_1) \ge q_i(t_1)$ for $i \in \two$ given $p_i(0) \ge q_i(0)$.
\item  
\revision{Next, suppose that the first event is an arrival. For contradiction, assume
that $p_j(t_1)<q_j(t_1)$ for some $j\in\{1,2\}$. Since 
$p_j(0)\ge q_j(0)$ and only
one arrival occurs at time $t_1$, we must have
\begin{align}
    p_j(0)=q_j(0), \label{eqn:p-j-0-eq-q}
\end{align}
and at time $t_1$,
the arrival must be routed to queue $3-j$ in the $p$-system and to queue $j$
in the $q$-system.
Thus, the JSQ rule implies
\[
p_{3-j}(0) \leq p_j(0)
\qquad \text{and} \qquad
q_j(0) \leq q_{3-j}(0).
\]

Combining these inequalities with $\p(0) \ge  \q(0)$ and~\eqref{eqn:p-j-0-eq-q}, we obtain
\[
p_{3-j}(0)
\geq q_{3-j}(0)
\geq q_j(0)
= p_j(0)
\geq p_{3-j}(0).
\]
Every inequality in the above chain must therefore be an equality. Thus,
\[
p_1(0) = p_2(0)
\qquad \text{and} \qquad
q_1(0) = q_2(0).
\]
By the deterministic tie-breaking rule,
at $t_1$,
both systems must route the arrival to queue $1$, contradicting the fact that the arrival is routed to queue
$j$ in one system and to queue $3-j$ in the other.

}
\end{enumerate}
Combining the above two cases, we   show $p_i(t_1) \ge q_i(t_1)$ for $i \in \two$.
Proceeding in the same
way,
we   show that  $p_i(t) \ge q_i(t)$ for any $t \ge 0$ and  $i \in \two$.
\\
Next,
we will  show that $r(t) \le r(0)$ holds for every $t \ge 0$.
We
fix a sample path.
Consider the processes $\N_a(t), \N_d^1(t),$ and $\N_d^2(t)$ in~\eqref{eqn:couple-def}.
Let $t$ denote the first time that either an arrival or a potential service completion  event
occurs.
\begin{enumerate}
    \item If the first event is an arrival, then it is obvious that $r(t) = r(0)$.
\item If the first event is a potential service completion, then 
we let $i$ denote the index such that $\N_d^i(t) = 1$.
Since $\p(0) \ge \q(0)$, 
we have $p_i(t) \ge q_i(t)$ as shown above.   
Thus,
there are two cases:
\begin{enumerate}
\item $\I{p_i(t-) > 0} = \I{q_i(t-) > 0}$, which implies
$r(t) = r(0)$.
\item $p_i(t-) > 0,q_i(t-) = 0$, and $dp_i(t) = -1,dq_i(t) = 0$,  
which 
implies that $r(t) \le r(0)$.  
\end{enumerate}   
\end{enumerate}  
Combining the two cases, we show  $r(t) \le r(0)$
for the $t$ defined above.
Using the same argument, we show  that
 $r(t) \le r(0)$ holds for every $t \ge 0$.

%% file: proof_lemma_T_x.tex
\subsection{Proof of Lemma~\ref{lemma-T-x}}
\label{sec:proof-lemma-T-x}




Consider two  coupled JSQ  systems 
with queue length processes $\p(t)$ and  $\q(t)$
as governed by~\eqref{eqn:couple-def}.
Suppose $\p(0)=(x+1,y)$ and $\q(0)=(x,y)$ for some $x \ge 0, y \ge 0$. 
From Lemma~\ref{lemma-dominance-JSQ},
for any $t \ge 0$, $\q(t) = (0,0)$ implies 
$\p(t) \in \set{(1,0),(0,1), (0,0)}$. 

Next, we show that
$T(0,1) \deq  T(1,0) \deq T_0$. 
We couple two copies of the JSQ system, with queue-length processes
$\q^{(1)}$ and $\q^{(2)}$, started from $(1,0)$ and $(0,1)$,
respectively. We use the same arrival process in both systems and couple
the service-completion times of their unique initial jobs to be identical.
Thus, if the service completion occurs before the first arrival, both
systems move to $(0,0)$; if an arrival occurs first, both systems move to
$(1,1)$. Hence, under this coupling, the two copies are in the same
post-jump state at the first event time.
From that time onward, we couple them using the same
arrival and potential service processes, so that they evolve identically.
Thus,  $T(0,1) \deq  T(1,0) \deq T_0$.

Combining $\q(t) = (0,0) \Rightarrow
\p(t) \in \set{(1,0),(0,1), (0,0)}, \forall t \ge 0$, 
$\quad T(0,1) \deq  T(1,0) \deq T_0$, and
the strong Markov property,  we have
$\T(x+1,y) \Leq \T(x,y) + T_0'$, where $T_0'$ is an independent copy of $T_0$. 
Similarly, we can show that
$\T(x,y+1) \Leq \T(x,y) + T_0'$. 
Iterating these one-step bounds yields
$T(x_1,x_2) \Leq \sum_{i=1}^{x_1+x_2} T_{0i}$.

%% file: proof_lemma_jsq_hitting_heavy_traffic.tex
\subsection{Proof of Lemma~\ref{lemma-JSQ-hitting-time-heavy-traffic}}

\label{sec:proof-lemma-JSQ-hitting-time-heavy-traffic}

\newcommand\mm{^{\text{M}}}
\newcommand\jsq{^{\text{J}}}
\newcommand{\Ppit}[1]{\mathbb{P}_{\pi\mm}\left(#1 \right)}

\arxiv{
As discussed, 
we   analyze
the expected   return time of a CTMC constructed by two 
$M/M/1$ queues.
First, we present 
a general result for the expected hitting time in terms of 
the $Q$-matrix of a CTMC.
\begin{proposition}[Theorem 3.5.3 of~\cite{Norris97}]
Let $Q$ be an irreducible $Q$-matrix of a CTMC. 
The following are equivalent: \\
(i) the CTMC is positive recurrent; \\ 
(ii) $Q$ is non-explosive and has a steady state $\pi$.
\\
Moreover, 
when (ii) holds, we have  $m_i = \frac{1}{\pi_i \nu_i}$ for all $i$,  
where 
$\nu_i$ is the absolute value of
$i$-th diagonal entry of $Q$ and
 $m_i$ is the expected   time 
 to 
 return to state $i$  
 when starting from state $i$
including the holding time in state $i$.
 \label{prop:norris}
 \end{proposition}
Then we use the following result comparing JSQ and a relevant  system
with two independent $M/M/1$ queues.
\begin{proposition}[Theorem 4 of~\cite{Turner98}] 
\label{prop-JSQ<=MM1}
  Let $(q\jsq_1(t),q\jsq_2(t))$ denote the queue length vector for the JSQ system defined in Theorem~\ref{thm:main}
  and 
 $(q\mm_1(t),q\mm_2(t))$ denote the queue length vector for the two independent $M/M/1$ queues  with $\bm{q}\jsq(0) = \bm{q}\mm(0)$. 
 Further, we assume both  $M/M/1$
 queues have an
 arrival rate  of $\frac{\lambda}{2}$ and a service rate  of $\mu$. 
 Then, there exists a coupling such that
 \begin{align*}
\sum_{i=1}^2 q\jsq_i(t)  \le
\sum_{i=1}^2 q\mm_i(t),   t \ge 0 \text{ almost surely}.
 \end{align*}
 This shows that, in this coupling, almost surely, 
 $\q\mm(t)=(0,0)$ implies $\q\jsq(t)=(0,0)$ for any $t \ge 0$. 
 \end{proposition}
 
With the above tool, we prove Lemma~\ref{lemma-JSQ-hitting-time-heavy-traffic} as follows.}
Consider two independent $M/M/1$ queues
with arrival rates $\lambda/2$ and service rates $\mu$, and denote their queue lengths 
as
$q\mm_1(t)$ and $q\mm_2(t)$.
Then their Cartesian product, $\bm{q}\mm(t) = (q\mm_1(t),q\mm_2(t))$,
is also a CTMC.  
 Since each marginal queue is irreducible and positive recurrent under the condition $\lambda/2 < \mu$, it follows that the joint process $\bm{q}\mm(t)$ is also irreducible and positive recurrent, with 
 the stationary distribution given by the product of the two marginal stationary distributions. 
 Under stationarity,
   $\P{q\mm_1 = k} = \P{q\mm_2 = k} = \rho^k(1-\rho), \forall k \in \mathbb{Z}_+$. 
Since they are independent, 
$\Ppit{q\mm_1=q\mm_2=0} = (1-\rho)^2$,
where $\pi\mm$ is the stationary distribution of $\q\mm$.

Let $m\mm_{(0,0)}$ denote the  
expected   time
for    $\bm{q}\mm$
to return to    $(0,0)$
when starting from $(0,0)$,
including the holding time in $(0,0)$.
Applying \journal{Theorem~3.5.3 of~\cite{Norris97}}\arxiv{Proposition \ref{prop:norris}}
to $\bm{q}\mm$,
and noting that the total transition rate out of state $(0,0)$ is
$\nu\mm_{(0,0)}=\lambda$, we obtain
\begin{align}
m\mm_{(0,0)}
= \frac{1}{\nu\mm_{(0,0)} \Ppit{q\mm_1=q\mm_2=0}} 
= \frac{1}{\lambda (1-\rho)^2}.
\label{eqn:m-0-0-exp}
\end{align}

Starting from state $(0,0)$, the process   $\bm{q}\mm$
leaves $(0,0)$ after an $\mathsf{Exp}(\lambda)$ time and
moves to either $(1,0)$ or $(0,1)$. 
Let $T_{(1,0)}^{(0,0)}$ be the time
for $\bm{q}\mm$
to hit      state $(0,0)$
starting from state $(1,0)$.
By 
symmetry,
\begin{align*}
    m\mm_{(0,0)} =
    \frac{1}{\lambda}+\E\left[T^{(0,0)}_{(1,0)}\right].
\end{align*}
Combining the above equation 
with \eqref{eqn:m-0-0-exp}, we have
\begin{align}
\E\left[T^{(0,0)}_{(1,0)}\right]
=
 \frac{1}{\lambda} \left(\frac{1}{(1-\rho)^2} - 1\right)
= \frac{1}{\lambda} \frac{\rho (2 - \rho)}{(1-\rho)^2}.
\label{eqn:E-T-0-0-1-0-expr}
\end{align}


Next, consider a two-queue JSQ system $\bm{q}\jsq(t)=(q\jsq_1(t),q\jsq_2(t))$ 
(with an arrival rate $\lambda$ and  service rate $\mu$ at each queue)
and the product chain
$\bm{q}\mm(t)=(q\mm_1(t),q\mm_2(t))$ of two independent $M/M/1$ queues defined previously,  both
started from the same initial state
$\q\jsq(0)=\bm{q}\mm(0)=(1,0)$. 
\journal{By Theorem~4 of~\cite{Turner98}, there exists a coupling of
$\{\q\jsq(t)\}_{t\ge0}$ and $\{\bm{q}\mm(t)\}_{t\ge0}$ such that  
\begin{equation}\label{eq:turner-majorization}
\sum_{i=1}^2 [q\jsq_i(t)-x]_+ 
\le
\sum_{i=1}^2 [q\mm_i(t)-x]_+, \forall x \ge 0, t \ge 0,
\end{equation}
where $[x]_+ = \max\{x,0\}$.
In particular, taking $x=0$ in \eqref{eq:turner-majorization}
yields $\sum_{i=1}^2 q\jsq_i(t)\le \sum_{i=1}^2 q\mm_i(t)$ for all $t\ge0$.
It follows that}\arxiv{By Proposition \ref{prop-JSQ<=MM1}, there exists a coupling such that}
\begin{align*}
T_0  = \inf\set{t \ge 0:  \bm{q}\jsq(t)=(0,0)}
\le
\inf\set{t \ge 0:  \bm{q}\mm(t)=(0,0)}
=
T^{(0,0)}_{(1,0)}
\qquad\text{almost surely.}
\end{align*}
Using the above inequality and \eqref{eqn:E-T-0-0-1-0-expr}, we show  $\E[T_0] \le
\E\left[T^{(0,0)}_{(1,0)}\right]
=
\frac{1}{\lambda} \frac{\rho (2 - \rho)}{(1-\rho)^2}$
and we finish the proof.

%% file: proof_improved_bound.tex
\subsection{Complete proof of Theorem~\ref{thm:main}}
\label{sec:improved-K-rho}

 To prove
Theorem~\ref{thm:main},
we will construct  a  tighter upper bound on $\E[T_0]$
in Lemma~\ref{lemma-JSQ-T0T1}.
Throughout,   when we write $Z,Z'$, we
 mean random variables that are independent copies of $Z$,
  and are independent of all random variables defined previously (unless stated otherwise).

 First,  we define $T_1$ as 
\begin{align}
T_1 =&  \inf \set{ t \ge 0:  \q(t) \in \set{(0,1), (1,0)} \mid \q(0) = (1,1)}
\label{eqn:T1-def}.
\end{align}

\begin{lemma} \label{lemma-JSQ-T0T1}  
Define two independent random variables
$X \sim \Exp{\lambda + 2\mu}$
and  $Y \sim \Exp{\lambda + \mu}$,
then  $T_1 \Leq A$ and $T_0 \deq B$ where $A$ and $B$ are random variables defined as
  \begin{align}
A \deq X +    \left\{
\begin{aligned}
 & 0  &, ~ \wp ~ \frac{2\mu}{\lambda + 2\mu}, \\
 &  T_1' + T_0'  &, ~ \wp ~ \frac{\lambda}{\lambda + 2\mu}, \\
\end{aligned}
\right.
\label{eqn:rv-A-def}
\end{align}

and

\begin{align}
B  \deq Y +   \left\{
\begin{aligned}
& 0 &, ~ \wp ~ \frac{\mu}{\lambda + \mu}, \\
&   T_1' + T_0'  &, ~ \wp ~ \frac{\lambda}{\lambda + \mu}, \\
\end{aligned}
\right.
\label{eqn:rv-B-def}
\end{align}
where
$T_1' \deq T_1, T_0' \deq T_0$,   $T_1',T_0'$ are independent, and 
independent of $T_1,T_0,X,Y$.
 Furthermore, if $\rho < \frac{1}{\sqrt{2}}$, 
   \begin{align}
 \left\{
\begin{aligned}
 & \E[T_1] \leq  &  \frac{\lambda+\mu}{2\mu^2  - \lambda^2}, \\
 & \E[T_0] \leq   &  \frac{\lambda+2\mu}{2\mu^2  - \lambda^2}.  
 \end{aligned}
\right. \label{eqn:ET1-T0-bound}
\end{align}
Further,
\begin{align}
\E[T_0] \le \ \frac{\lambda+2\mu}{2\mu^2  - \lambda^2}
 = \frac{1}{\lambda} \frac{1 + \frac{2\mu}{\lambda}}{2 \frac{\mu^2}{\lambda^2} - 1}
  = \frac{1}{\lambda} \frac{1+ \frac{1}{\rho}}{\frac{1}{2} \frac{1}{\rho^2} - 1}
=\frac{1}{\lambda} \frac{2\rho^2 + 2\rho}{1-2\rho^2}
. 
\label{eqn:algebraic-for-K-rho}
\end{align}
 \end{lemma}



We replace  \eqref{eqn:ET1-T0-bound-general}
in the proof of Lemma~\ref{lemma-JSQ-T-prob}
by the sharper bound \eqref{eqn:algebraic-for-K-rho}
from   Lemma~\ref{lemma-JSQ-T0T1} (and taking the minimum with \eqref{eqn:ET1-T0-bound-general} when appropriate).
Recalling the expression of $K(\rho)$ in \eqref{eqn:def-K},
since
\begin{align*}
\frac{\rho (2 - \rho)}{(1-\rho)^2}
\cdot \left( 
\frac{2+ 2\rho}{1-2\rho^2} 
\cdot \frac{(1-\rho)^2} {2- \rho}\right)
= \frac{2\rho^2 + 2\rho}{1-2\rho^2},
\end{align*}
we
show $
\P{T \ge t} \le
\frac{1}{t} C^K_1(\lambda,\mu,\x),$
where $C^K_1(\lambda,\mu,\x)$ is defined in~\eqref{eq:definition of C1-K}.
Thus, we complete the
proof of Theorem~\ref{thm:main}.

%% file: proof_lemma_JSQ_T0_T1.tex
\subsection{Proof of Lemma~\ref{lemma-JSQ-T0T1}}
\label{sec-proof-lemma-JSQ-T0T1}

To prove  Lemma~\ref{lemma-JSQ-T0T1}, we   use   Lemma~\ref{lemma-dominance-JSQ} 
to construct a recursive relationship  
between 
  $T_0$~\eqref{eqn:T0-def} 
and $T_1$~\eqref{eqn:T1-def}.
We define
\begin{align*}
\T(\x; \mc{A}) \defeq  \inf
\set{ t  \ge 0: 
\p(t) \in \mc{A}
\mid \p(0) = \x 
}
\end{align*} as 
the random time it takes 
for a JSQ queue length process
$\p(t)$ to hit a set $\mc{A}$
starting from initial queue length $\x$.
 We analyze $T_1$~\eqref{eqn:T1-def} and $T_0$~\eqref{eqn:T0-def} separately.
\begin{enumerate}
\item 
We first analyze $T_1$~\eqref{eqn:T1-def}.
We let $\p(0) = (1,1)$,
and let $t_1$ denote the first time that 
there will be either an arrival or a service completion event. 
Then 
$t_1 \deq X$.
\begin{enumerate}
\item  If the first event is a service completion, which has probability $ \frac{2\mu}{\lambda + 2\mu}$, then  
$\p(t_1) \in \set{(1,0), (0,1)}$. 
\item If the first event is an arrival, which has probability $ \frac{\lambda}{\lambda + 2\mu}$, 
then $\p(t_1) = (2,1)$. 
Now consider two JSQ systems with queue length vectors $\s(t), \w(t)$ and $\s(0)=(2,1), \w(0) = (1,1)$. 
We also assume they are coupled with   arrival and   service, as defined in Section~\ref{sec:couple-definition}.
Thus,  $\w$ will hit
the set
$\set{(1,0), (0,1)}$ 
at some time $t_2 \deq T_1$ 
with $T_1$ defined in~\eqref{eqn:T1-def}.
By Lemma~\ref{lemma-dominance-JSQ},
\begin{align*}
\s(t_2) \in \set{(1,1),(1,0),(0,1),(2,0),(0,2)}.  
\end{align*}
Next, we analyze the total
time for $\s$ to hit $\set{(1,0),(0,1)}$ starting from its initial state $\s(0)=(2,1)$, which can be written as 
$\mc{T}((2,1); \set{(1,0), (0,1)})$.
There are several cases:
\begin{enumerate} 
\item      
If $\s(t_2) \in \set{(1,0), (0,1)}$,
then we have $\mc{T}((2,1); \set{(1,0), (0,1)}) = t_2$ conditioned on this event.
\item     
If $\s(t_2) = (1,1)$,
and suppose it
takes
another $T_1'$ time 
for $\s$ to hit 
$\set{(1,0), (0,1)}$.
In addition,
using 
Lemma~\ref{lemma-dominance-JSQ},
we can show $T_1' \Leq T_0'$:
indeed, couple $\s$ (started from $(1,1)$) with a JSQ process $\q$ started from $(1,0)$
using the common arrival/service construction, and let
\[
\tau := \inf\{t\ge 0 : \q(t) = (0,0)\} \stackrel{d}{=} T_0'.
\]
By Lemma~\ref{lemma-dominance-JSQ}, $\q(\tau)=(0,0)$ implies
$\s(\tau)\in \set{(0,0),(1,0),(0,1)}$, so $\s$ must have hit $\set{(1,0),(0,1)}$ by time $\tau$.
Thus, $T_1' \le \tau$ almost surely, which yields $T_1' \Leq T_0'$.
\item     
If $\s(t_2) \in \set{(2,0),(0,2)}$.
Couple a JSQ process $\p$ with $\p(0)=(2,0)$ and another JSQ process $\q$ with $\q(0)=(1,0)$ using the common arrival/service construction, and let $\tau:=\inf\{t \ge 0: \q(t)=(0,0)\}\overset d=T_0$. By Lemma~\ref{lemma-dominance-JSQ}, $\q(\tau)=(0,0)$ implies $\p(\tau)\in\{(0,0),(1,0),(0,1)\}$; moreover, any sample path from $(2,0)$ to $(0,0)$ must hit $(1,0)$ or $(0,1)$ beforehand. Thus, $\mc{T}((2,0); \set{(1,0), (0,1)})  \le \tau$ a.s., and therefore $\mc{T}((2,0); \set{(1,0), (0,1)}) \preceq T_0$ (the case $s(t_2)=(0,2)$ follows by symmetry).
\end{enumerate}

In summary, the above argument shows that 
\begin{align*}
  \mc{T}((2,1); \set{(1,0), (0,1)}) \Leq T_1 + T_0. 
\end{align*}
Therefore, $T_1 \Leq A$ with $A$ defined in \eqref{eqn:rv-A-def}.
\end{enumerate}
 \item 
Now we   analyze $T_0$~\eqref{eqn:T0-def}.
Let $\p(0)= (1,0)$,
and let $t_1$ denote the first time at 
which either an arrival or a service completion event occurs.
Then 
$t_1 \deq Y$.
\begin{itemize}
\item      
If the first event is a service completion, which has probability 
$ \frac{\mu}{\lambda + \mu}$, 
then $\p(t_1) = (0,0)$.
\item 
If the first event is an arrival, which has probability $\frac{\lambda}{\lambda + \mu}$, 
then $\p(t_1) = (1,1)$,
and  the time it takes for the queue length vector
$\p(t)$ 
to hit $(0,0)$,
which we denote as $\mc{T}((1,1);(0,0))$, satisfies
\begin{align*}
\mc{T}((1,1);(0,0)) \deq T_1' + T_0'.
\end{align*}
Thus, we have shown
 $T_0 \deq B$
with $B$ defined in \eqref{eqn:rv-B-def}.
\end{itemize}   
\end{enumerate}

Combining $T_1 \Leq A$, and
$T_0 \deq B$,
we arrive at
\begin{align*}
\left\{
\begin{aligned}
& \E[T_1] \leq  & \frac{1}{\lambda + 2\mu} +  \frac{\lambda}{\lambda + 2\mu} \left(\E[T_1]+\E[T_0]\right), \\
& \E[T_0] =  & \frac{1}{\lambda + \mu} +  \frac{\lambda}{\lambda + \mu} \left(\E[T_1]+\E[T_0]\right).  
\end{aligned}
\right.
\end{align*}

It follows that $2\mu \E[T_1] \leq 1+ \lambda \E[T_0] = 1+ \frac{\lambda}{\mu} + \lambda \frac{\lambda}{\mu} \E[T_1]
$ and we know this system is positive recurrent, 
which implies that $\E[T_1] < \infty$ and $\E[T_0] < \infty$.
Solving the above inequality yields~\eqref{eqn:ET1-T0-bound}
 if $\rho < \frac{1}{\sqrt{2}}$.

%% file: proof_of_corollary.tex
\subsection{Proof   of Corollary~\ref{coro-moment-difference}}
\label{sec:proof-coro}

To show Corollary~\ref{coro-moment-difference}, we will 
derive an upper bound for
$\E[q_1(t)^2 + q_2(t)^2]$, where $\q(t)$ is the queue length process for a JSQ system.
To do this,
\journal{since under some coupling,
the   queue length vector of a two-queue JSQ
system
can be component-wise upper bounded by 
that of two independent $M/M/1$ systems~\citep{Turner98},}\arxiv{due to Proposition~\ref{prop-JSQ<=MM1},}
it suffices to analyze two independent $M/M/1$ systems.
Thus, in Lemma~\ref{prop-MM1-upper-bound}, we analyze the corresponding $M/M/1$ systems.

\begin{lemma}  \label{prop-MM1-upper-bound}
Consider an $M/M/1$ queue with an arrival rate $\lambda$ and a service rate $\mu$. Let $q(t)$ be its queue length at time $t$ and assume $\rho = \frac{\lambda}{\mu} < 1$.
Let $\pi$ denote its steady state.
Then 
\begin{align*}
\E_\pi[q^2]=
\frac{\rho(1+\rho)}{(1-\rho)^2},
\end{align*}

and 
\begin{align*}
\E[q(t)^2|q(0) = x]  \leq x^2 +  2 x \frac{\rho}{1-\rho}    
+\frac{\rho(1+\rho)}{(1-\rho)^2}.
\end{align*}
\end{lemma}


\begin{proof}[Proof of Lemma~\ref{prop-MM1-upper-bound}]
\label{sec-proof-MM1-upper-bound}
The first equality follows from the fact that  the steady-state 
distribution of an $M/M/1$ system is $\pi(k) = \rho^k (1-\rho)$ so $\E_\pi[q^2]
=\sum_{k=0}^\infty k^2 \rho^k (1-\rho)
= \frac{\rho(1+\rho)}{(1-\rho)^2}
$.

  Let $q_x(t)$ denote the queue length process starting from $q(0) = x \ge 0$, and let $q_0(t)$ denote the  queue length process starting from $q(0) = 0$.
We couple the two processes by using the same arrival and potential 
service processes: both queues receive arrivals 
from a common Poisson process of rate~$\lambda$ and services from a common Poisson process of rate~$\mu$.
Under this coupling, the process starting from $x$ remains below the one starting from $0$ plus $x$, that is,
\begin{align*}
q_x(t) \le q_0(t) + x \text{ almost surely}   \text{ for all } t \ge 0.
\end{align*}
Squaring both sides and taking expectations, we obtain
\begin{align*}
\E[q_x(t)^2]
&\le \E[(q_0(t) + x)^2] \\
&= \E[q_0(t)^2] + 2x\, \E[q_0(t)] + x^2.
\end{align*}
Thus,
\begin{align}
    \E[q(t)^2 \mid q(0) = x]
    &= \E[q_x(t)^2] \nonumber \\
    &\le \E[q_0(t)^2] + 2x\, \E[q_0(t)] + x^2  \nonumber\\
&= \E[q(t)^2 \mid q(0) = 0] + 2x\, \E[q(t) \mid q(0) = 0] + x^2. 
\label{eqn:final-Eq}
\end{align}

In addition, let $q_\pi(t)$
denote the  queue length process starting from $q(0) \sim \pi$. We couple the two processes by using the same arrival and service processes.
Since $q_0(0) \le q_\pi(0)$  almost surely,
\begin{align*}
q_0(t) \le q_\pi(t) \text{ almost surely}   \text{ for all } t \ge 0.
\end{align*}
Thus,
\begin{align}
\E[q(t)^2 \mid q(0) = 0] & \le \E_\pi[q^2]=
\frac{\rho(1+\rho)}{(1-\rho)^2}, \label{eqn:E-qt-2} \\
\E[q(t)  \mid q(0) = 0] & \le \E_\pi[q] = \frac{\rho}{1-\rho}.
\label{eqn:E-qt} 
\end{align}

Combining~\eqref{eqn:E-qt-2},~\eqref{eqn:E-qt}, and~\eqref{eqn:final-Eq}, we have
\begin{align*}
\E[q(t)^2|q(0) = x] \leq x^2 +  2 x 
\frac{\rho}{1-\rho} 
+ \frac{\rho(1+\rho)}{(1-\rho)^2},
\end{align*}
and we  finish the proof.
\end{proof}

Now we are ready to prove  Corollary~\ref{coro-moment-difference}.
\begin{proof}[Proof of Corollary~\ref{coro-moment-difference}]

Fix an integer $R >0$.
Then
\begin{align}
\begin{split}
&|\E[q_1(t)] - \E_\pi[q_1]|    \\ 
& = |\E[q_1(t) \I{q_1(t) \leq R}]
+\E[q_1(t) \I{q_1(t) > R}]
- \E_\pi[q_1 \I{q_1 \leq R}]
-\E_\pi[q_1 \I{q_1 > R}]|   \\
& \leq  |\E[q_1(t) \I{q_1(t) \leq R}] -  \E_\pi[q_1 \I{q_1 \leq R}]| 
+|\E[q_1(t) \I{q_1(t) > R}] -  \E_\pi[q_1 \I{q_1 > R}]|  \\
&  \stackrel{(a)}{=}  \left|\sum_{m=0}^{R-1} \P{m < q_1(t) \le R}  - \Ppi{m < q_1 \le R}\right| 
+|\E[q_1(t) \I{q_1(t) > R}] -  \E_\pi[q_1 \I{q_1 > R}]|  \\
& \stackrel{(b)}{\le}
\sum_{m=0}^{R-1} 
|\P{m < q_1(t) \le R}
-\Ppi{m < q_1 \le R}| + \frac{1}{R} \Big(\E[q_1(t)^2 \I{q_1(t) > R}] +  \E_\pi[(q_1)^2 \I{q_1 > R}] \Big) \\
& \stackrel{(c)}{\le} R \cdot \TV{\mathbb{P}_{\x}^t}{\pi} + \frac{1}{R}  \Big(\E[q_1(t)^2] +  \E_\pi[(q_1)^2] \Big).
\end{split} \label{eqn:coro-imp-eq}
\end{align} 
In~\eqref{eqn:coro-imp-eq},
$(a)$ follows from the 
equality 
$\E[X \I{X \le a}] = \sum_{m=0}^{a-1} \P{m < X \le a}$
for any non-negative integer-valued random variable $X$  and $a \in \mathbb{N}_+$, 
$(b)$ follows from the triangle inequality, 
the fact that 
$x \I{x \ge R} \le \frac{x^2}{R} \I{x \ge R}$ for any $x \ge 0$, and $|x-y| \le x+ y$ for any $x,y \ge 0$,
and
$(c)$ follows from 
$|\mathbb{P} (A) - \mathbb{Q}(A)| \le \TV{\mathbb{P}}{\mathbb{Q}}$  with $A =
\set{(q_1,q_2): m < q_1 \le R}$. 
Similar to~\eqref{eqn:coro-imp-eq}, we have
\begin{align*}
|\E[q_2(t)] - \E_\pi[q_2]| \leq  R  \TV{\mathbb{P}_{\x}^t}{\pi} + \frac{1}{R}  \Big(\E[q_2(t)^2] +  \E_\pi[(q_2)^2] \Big).
\end{align*} 

The triangle inequality yields 

\begin{align*}
|\E[q_1(t)  + q_2(t)] - \E_\pi[q_1 + q_2]| \leq 
2R  \TV{\mathbb{P}_{\x}^t}{\pi}  + \frac{1}{R}   \Big(\E[q_1(t)^2 + q_2(t)^2] +  \E_\pi[(q_1)^2 + (q_2)^2] \Big). 
\end{align*}

Since $\rho = \frac{\lambda}{2\mu}$,
then by Lemma~\ref{prop-MM1-upper-bound} and  
Proposition~\ref{prop-JSQ<=MM1},
$\E_\pi[(q_1)^2+(q_2)^2] \leq \frac{2\rho(1+\rho)}{(1-\rho)^2}$, and 
\begin{align*}
\E[q_1(t)^2 + q_2(t)^2]  
\le
    (x_1^2+ x_2^2)  + 2(x_1+x_2)  
\frac{\rho}{1-\rho}
+ \frac{2\rho(1+\rho)}{(1-\rho)^2}. 
\end{align*} 
Taking $R = \ceil{\sqrt{t}} \le \sqrt{t} + 1$ yields
\begin{align*} 
 |\E[q_1(t) + q_2(t)] - \E_\pi[q_1 + q_2]|  
 \leq 
2 (\sqrt{t}+1) \TV{\mathbb{P}_{\x}^t}{\pi} +
\frac{1}{\sqrt{t}} 
C_2(\lambda,\mu,\x).
\end{align*}

By Theorem~\ref{thm:main}, $\TV{\mathbb{P}_{\x}^t}{\pi} \leq \frac{1}{t}  C_1^K(\lambda,\mu,\x)$, where $C_1^K$ is defined 
in~\eqref{eq:definition of C1-K}.

Therefore,  we conclude that

\begin{align*}
|\E[q_1(t)+q_2(t)] - \E_\pi[q_1 + q_2]| \leq \frac{1}{\sqrt{t}}  \Big(2C_1^K(\lambda,\mu,\x) + C_2(\lambda,\mu,\x)
\Big) + \frac{2}{t} C_1^K(\lambda,\mu,\x).
\end{align*}  
Using Proposition~\ref{prop-JSQ-stationary-bound}, we finish proving~\eqref{eqn:moment-bound}.
\end{proof}

%% file: conclusion.tex
\section{Conclusion and future work}
\label{sec:conclusion}

In this paper, we present a convergence rate result for a load balancing system with two queues under the JSQ policy,
which is valid for any $\rho \in (0,1)$.  
Our method is based on a 
novel  and simple
coupling construction 
and on computing an upper bound on the expected hitting time. 
Instead of relying on approximations of the stochastic systems in relevant limits, 
we directly analyze the original system, yielding clean insights regarding the convergence rate and how it depends on $\rho$.

Below, we discuss some interesting future directions.
Recall that we have established a convergence rate that is of order  
$\tvratet$. However, based on the results for a single-server queue \cite{Robert13} 
and also based on the scaling used in diffusion approximations, 
one expects an exponential convergence rate of the form $O\left( e^{-\alpha \epsilon^2 t}\right)$ 
with $\epsilon = 1- \rho$ for 
an appropriate constant $\alpha$. 
Establishing such a strong bound on convergence is a future research direction, and we believe
it requires adopting novel methodological approaches, particularly by exploiting
state space collapse results from the literature.
Another possible approach is to compute an %
upper bound on the moment generating function of the hitting time.
An alternative direction
is to extend our analysis to a system with $n \ge 3$
servers and heterogeneous service rates. 
We also hope the technique
we have developed in this paper
can spur further analysis
of more 
complicated
stochastic  systems.